\documentclass[11pt]{amsart}

\usepackage[margin=1in]{geometry}

\usepackage{amsmath}
\usepackage{amssymb}
\usepackage{amsfonts}
\usepackage{amsthm}
\usepackage{mathtools}
\usepackage{titlesec}
\titleformat{\section}
  {\normalfont\large\centering}
  {\thesection.}
  {0.5em}
  {}
\usepackage[english]{babel}

\usepackage{mathrsfs}
\usepackage{bm}

\usepackage{enumitem}

\usepackage{booktabs}
\usepackage{array}

\usepackage{graphicx}
\usepackage{tikz}
\usetikzlibrary{arrows.meta, positioning, calc, matrix}

\usepackage{algorithm}
\usepackage{algpseudocode}

\usepackage{xcolor}

\usepackage[hidelinks]{hyperref}

\usepackage{cleveref}

\usepackage{microtype}

\usepackage[numbers,sort&compress]{natbib}

\newtheorem{theorem}{Theorem}[section]
\newtheorem{lemma}[theorem]{Lemma}
\newtheorem{proposition}[theorem]{Proposition}
\newtheorem{corollary}[theorem]{Corollary}

\theoremstyle{definition}
\newtheorem{definition}[theorem]{Definition}

\theoremstyle{remark}
\newtheorem{remark}[theorem]{Remark}
\newtheorem{conjecture}{Conjecture}[section]
\usepackage{titlesec}
\titlespacing*{\section}{0pt}{10pt}{4pt}
\titlespacing*{\subsection}{0pt}{8pt}{3pt}
\titlespacing*{\subsubsection}{0pt}{6pt}{2pt}
\title{The Delta Set of a Numerical Monoid Generated by a Generalized Arithmetic Sequence}
\author{Ashwin Narayan}
\date{August 2026}

\begin{document}
\pagestyle{plain}

\maketitle

\begin{abstract}
Let
\[
S_h=\langle a,ah+d,ah+2d,\ldots,ah+kd\rangle,
\qquad \gcd(a,d)=1,\quad 1<k<a,\quad h> 1,
\]
be a numerical monoid generated by a generalized arithmetic sequence. In this paper, we determine the delta set \[\Delta(S_h),\] resolving an open problem posed by Christopher O'Neill and Roberto Pelayo in 2014 \cite{ONeill}. Writing $e=h-1$ and letting $E(e,d)$ denote the set of values arising from the subtractive Euclidean algorithm applied to $e$ and $d$, we prove that
\[
\Delta(S_h)
=
E(h-1,d)\cup
\{M-r(h-1):M-r(h-1)>0,\ r\geq0\},
\]
where
\[
M=\left\lceil\frac{a}{k}\right\rceil(h-1)+d
\]
is the maximal element of $\Delta(S_h)$. In particular, $d$ and $h-1$ always belong to $\Delta(S_h)$. We further establish a reduction theorem showing that $\Delta(S_h)$ is a multiple of the delta set of a primitive monoid satisfying $\gcd(d,h-1)=1$, and deduce that the minimum delta is $\gcd(d,h-1)$. Finally, we investigate the relationship between delta sets and catenary degrees for these monoids, obtaining an explicit connection between their maximal values and formulating several conjectures concerning the catenary degree set. These results provide a structural description of factorization invariants for generalized arithmetic numerical monoids and suggest several directions for further study.
\end{abstract}
\section{Introduction}

Factorization theory provides a framework for studying the arithmetic
structure of numerical monoids and their non-unique factorizations.
For general background on numerical semigroups, see
Rosales and Garc\'ia-S\'anchez~\cite{RosalesGarcia}. A numerical
monoid is an additive submonoid of $(\mathbb{N}_0,+)$ of the form
\[
S=\langle n_1,n_2,\ldots,n_k\rangle,
\qquad
\gcd(n_1,\ldots,n_k)=1.
\]
In non-unique factorization theory, an element of a numerical semigroup may have multiple factorizations,
and these factorizations need not have the same length. The
\emph{delta set} records the successive differences between the
distinct factorization lengths of elements of the semigroup, providing a key measure of the extent of non-unique
factorization.

In this paper, we study the delta sets of the generalized arithmetic
numerical monoids
\[
S_h=
\langle
a,\,
ah+d,\,
ah+2d,\,
\ldots,\,
ah+kd
\rangle,
\]
where
\[
\gcd(a,d)=1,
\qquad
1<k<a,
\qquad
h> 1.
\]
These monoids form a natural generalization of numerical monoids
generated by arithmetic sequences and have been studied extensively
in the context of factorization invariants.

Our main result gives an explicit characterization of the delta set
in the range $1<k<a$. Previous work has established a number of
properties of delta sets for generalized arithmetic sequences and
related families; see, for example,
Bowles and Chapman~\cite{BowlesChapman}, and
Garc\'ia, Llena, and Moscariello~\cite{GarciaLlenaMoscarielloSym}.
However, an explicit description of the entire delta set in the
regime $1<k<a$ has remained unresolved.

We show that the delta set decomposes into two distinct components:
an arithmetic component determined by $a,h,d,k$, and a Euclidean
component determined by $d$ and $h-1$. More precisely, define
\[
M=\left\lceil\frac{a}{k}\right\rceil(h-1)+d.
\]
Then our main theorem gives
\[
\Delta(S_h)
=
E(d,h-1)
\cup
\left\{
M-r(h-1):
r\geq 0,\;
M-r(h-1)>0
\right\},
\], or equivalently $E(M,e)$
where $E(d,h-1)$ denotes the Euclidean component arising from the
subtractive Euclidean algorithm applied to $d$ and $h-1$.

In particular, the largest delta is given by
\[
\boxed{
\max\Delta(S_h)
=
\left\lceil\frac{a}{k}\right\rceil(h-1)+d.
}
\]
The maximal delta is governed by the arithmetic structure of
the generating sequence, while the remaining delta values exhibit
a Euclidean structure arising from the interaction between $d$ and
$h-1$.

We first explore
the arithmetic structure of the generators to parametrize
factorizations by lattice points in $\mathbb{Z}^2$. This transforms
the problem of determining factorization lengths into the study of
the geometry of a two-dimensional lattice. We then analyze the
resulting lattice through its column structure. The Three-Gap
Theorem provides the key tool necessary for understanding the
Euclidean gaps that arise between successive factorization lengths. Together,
these tools allow us to identify the two components of the delta
set and to establish the decomposition.
\section{Paper Organization}
This part of the paper presents the overall architecture through which we
proceed with the main proof. Section~\ref{sec:preliminaries} introduces
the necessary definitions and background on factorizations, length sets,
delta sets, numerical monoids, and the Euclidean Algorithm.
Section~\ref{sec:Embeddingdimension3} demonstrates how the proposed
formula holds in embedding dimension $3$. Section~\ref{sec:lattice}
translates the original $(k+1)$-dimensional factorization space into a
two-dimensional lattice parametrization. Section~\ref{sec:column}
develops the column structure of this lattice, describing how the
lattice points decompose into columns and how their corresponding
factorization lengths behave. Section~\ref{sec:maximaldelta} establishes
the maximal delta element, which is then used in
Section~\ref{sec:arithmetic} to prove that
\[
\{M-re:r\geq 0,\ M-re>0\}\subseteq\Delta(S_h).
\]
Section~\ref{sec:Euclidean} proves that the Euclidean set is also
contained in the delta set, while Section~\ref{sec:otherdelta} analyzes
the geometry of truncated columns, for which
$U_j\neq\mu_s+aj$. Section~\ref{sec:forwardcontainment} proves the
reverse containment by showing that no other delta can occur between
consecutive factorization lengths. Finally,
Section~\ref{sec:maintheorem} combines the results of
Sections~\ref{sec:arithmetic},~\ref{sec:Euclidean}, and ~\ref{sec:forwardcontainment} to establish the main theorem.
\section{Prerequisites}
\label{sec:preliminaries}
Throughout this paper, let $S_h = \langle a,ah+d,ah+2d,ah+3d,...ah+kd \rangle $, be a numerical monoid generated by a generalized arithmetic sequence, where \[gcd(a,d) =1, 1<k<a, h > 1\] 
For convenience,we set $e =h-1.$
\\
Suppose the set of factorizations of $n \in S$ is given by
\[
Z(n) = \left\{ (z_1, z_2, \dots, z_p) \in \mathbb{N}^p \mid z_1 n_1 + z_2 n_2 + \dots + z_p n_p = n \right\}
\]
If $z = (z_1, z_2, z_3, \dots, z_p)$ is a factorization of $n$, then the length of $z$ is given by
\[
|z| = z_1 + z_2 + z_3 + \dots + z_p
\]
The set of lengths of $n$ is defined by 
\[L(n) = \{|z|\ |   z \in Z(n)\} \subset \mathbb{N} \]
\begin{definition}
We define the delta set as \[\Delta (S_h) = \{ l_2 -l_1 : l_2,l_1 \in L(s), L(s) \cap (l_1,l_2) = \emptyset \} \]
\end{definition}
\begin{definition}
We define the subtractive Euclidean Algorithm in this way: Let \[x_0 \ and \ y_0 \ = (e,d)\]
Recursively define 
\[
(x_{i+1},y_{i+1})=
\begin{cases}
(x_i-y_i,y_i),& x_i>y_i,\\
(x_i,y_i-x_i),& y_i>x_i,\\
\text{terminate},& x_i=y_i.
\end{cases} \]
\end{definition}

\section{Embedding Dimension 3}
\label{sec:Embeddingdimension3}
We first prove that this case holds for embedding dimension three.
\begin{remark}[The case $k=2$]
The case $k=2$ gives
\[
S_h=\langle a,ah+d,ah+2d\rangle .
\]
 In this case, Theorem~\ref{thm:maintheorem} gives the same description of the delta set obtained by Garc\'ia-S\'anchez,
Llena, and Moscariello for symmetric and nonsymmetric numerical semigroups
of embedding dimension three
\cite{GarciaLlenaMoscarielloSym,GarciaLlenaMoscarielloNonSym}. Set\ \(
e=h-1
\ \text{and} \ 
M=\left\lceil\frac{a}{2}\right\rceil e+d.
\)
Then Theorem~\ref{thm:maintheorem} gives
\[
\Delta(S_h)
=
E(e,d)\cup
\{M-re:M-re>0,\ r\geq0\}, \  \textnormal{or equivalently,} \ E(M,e)
\]

Comparing this with the results of
\cite{GarciaLlenaMoscarielloSym,GarciaLlenaMoscarielloNonSym}, we consider the
relation lattice of $S_h$. Using their framework,
\[
c_1=d+h\left\lceil\frac{a}{2}\right\rceil,
\qquad
c_2=2,
\qquad
c_3=\left\lceil\frac{a}{2}\right\rceil,
\]
with
\[
r_{12}=r_{32}=a\bmod 2,\qquad
r_{13}=\left\lfloor\frac{a}{2}\right\rfloor,
\]
\[
r_{21}=h,\qquad
r_{23}=1,\qquad
r_{31}=h\left\lfloor\frac{a}{2}\right\rfloor+d.
\]
In particular,
\[
c_1n_1=c_3n_3
\quad\Longleftrightarrow\quad
a\equiv0\pmod 2.
\]
$S_h$ is then symmetric if and only if $a$ is even. Hence, both the
symmetric and nonsymmetric cases are described by the two results of
Garc\'ia-S\'anchez, Llena, and Moscariello
\cite{GarciaLlenaMoscarielloSym,GarciaLlenaMoscarielloNonSym}.

Their Euclidean-algorithm invariants collapse into
\[
(\delta_1,\delta_2)=(M,e)
\qquad\text{if $a$ is even},
\]
while
\[
(\delta_1,\delta_3)=(M,M-e)
\qquad\text{if $a$ is odd}.
\]
Although the two pairs appear different, they produce the same delta set, since

\[
M\geq 2e+d>2e,
\]
so the first subtractive Euclidean step applied to $(M,M-e)$ gives
\[
(M,M-e)\longmapsto(M-e,e).
\]
Thus the subsequent Euclidean algorithm is the same as that obtained from
$(M,e)$.

It remains to compare this Euclidean description with the decomposition
in Theorem~\ref{thm:maintheorem}. Since
\[
M=\left\lceil\frac{a}{2}\right\rceil e+d
\equiv d\pmod e,
\]
the subtractive Euclidean algorithm applied to $(M,e)$ first produces the
values
\[
\{M-re:M-re>0,\ r\geq0\},
\]
until the remainder modulo $e$ is reached. The remaining Euclidean steps
are precisely those obtained from $(e,d\bmod e)$. Therefore
\[
\Delta(S_h)
=
\{M-re:M-re>0,\ r\geq0\}
\cup E(e,d\bmod e).
\]
Since the subtractive Euclidean algorithm is unchanged by replacing $d$
with its residue modulo $e$, 
\[
E(e,d\bmod e)=E(e,d).
\]
Hence the Euclidean-algorithm description in
\cite{GarciaLlenaMoscarielloSym,GarciaLlenaMoscarielloNonSym}
agrees exactly with Theorem~\ref{thm:maintheorem}.

For example, let
\[
a=4,\qquad h=143,\qquad d=109.
\]
Then
\[
S_h=\langle4,681,790\rangle,
\qquad e=142,
\qquad M=393.
\]
The arithmetic-progression component is
\[
\{M-142r:M-142r>0\}
=
\{393,251,109\},
\]
while the remaining values are obtained from the subtractive Euclidean
algorithm applied to $(109,142)$. This gives
\[
\Delta(S_h)
=
\{1,2,3,4,7,10,13,23,33,43,76,109,142,251,393\},
\]
in agreement with \cite[Example 38]{GarciaLlenaMoscarielloSym}.
\end{remark}


\section{Lattice Parametrization}
\label{sec:lattice}
It now remains to prove Theorem \ref{thm:maintheorem}.
\\
We proceed by first translating the dimension of the factorization space from
\[\mathbb{N}^{k+1}\]
into a more manageable 2D-plane.

Let $s\in S_h$. Then
\[s=ax+\sum_{i=1}^{k}(ah+id)y_i,\]
where $x$ represents the number of copies of $a$ in the factorization,
and $y_i$ represents the number of copies of $ah+id$ in the
factorization.

Define
\[m=\sum_{i=1}^{k}iy_i, \qquad y=\sum_{i=1}^{k}y_i.\]

\begin{lemma}
There exists
\[\mu_s\in\{0,1,2,\ldots,a-1\}\]
such that $Z(s)$ is in bijection with
\[R_s=\left\{(j,y)\in\mathbb{N}^2:L_j\leq y\leq U_j\right\},\]
where
\[L_j= \left\lceil \frac{\mu_s+aj}{k} \right\rceil, \]
and \(U_j= \min\left\{\mu_s+aj,\left\lfloor\frac{s-d(\mu_s+aj)}{ah} \right\rfloor\right\}.\)

We define the factorization length
$l_s(j,l)=C_s-d\mu_s-ey,$ where \(C_s=\frac{s-d\mu_s}{a}.\)
\end{lemma} 

\begin{proof}

Since
\[s=ax+\sum_{i=1}^{k} (ah+id)y_i,\]
this can be simplified as
\[s=ax+ah\sum_{i=1}^{k}y_i +d\sum_{i=1}^{k}iy_i =ax+ahy+dm.\]

This equation implies \(s\equiv dm\pmod a.\) Furthermore, since
\(\gcd(a,d)=1,\)
there exists some
\(\mu_s\in\{0,1,\ldots,a-1\}\)
such that
\(s\equiv d\mu_s\pmod a.\) Therefore,
\(m\equiv\mu_s\pmod a.\) So, \(m=aj+\mu_s, j\geq0.\)
Connecting m to \(y\)
via
\(
y\leq m\leq ky,
\)
we have
\(\left\lceil\frac{m}{k}\right\rceil \leq y\leq m.\) Additionally, since
\(
x\geq0,
\)
\(s\geq ahy+dm\implies \frac{s-dm}{ah}\geq y,\) creating yet another upper bound on $y.$ 
Since \[m=\mu_s+aj,\] the lower bound of $y$ then becomes
\[
\left\lceil \frac{\mu_s+aj}{k} \right\rceil,
\]
while the upper bound becomes
\[
\min\left\{
\mu_s+aj,\,
\left\lfloor
\frac{s-d(\mu_s+aj)}{ah}
\right\rfloor
\right\}.
\]
This proves both the upper and lower bounds of $y$. Solving for x in $s=ax+ahy+dm$ and substituting that in, 
\[l_s=x+y= \frac{s-dm-ahy}{ah}+y = \frac{s-d(\mu_s+aj)-ahy}{ah}+y.\]
Substituting $C_s,$
\[l_s= C_s-dj-ey.\]

\end{proof}


\section{Column Structure}
\label{sec:column}
We next consider the column structure for the lattice parametrization.
In particular, for a fixed
\[s\in S_h,\]
we define
\[C_j=C_s-dj-ey: \qquad L_j\leq y\leq U_j.\]

We call Column $j$ \emph{nonempty} if
\(
L_j\leq U_j,
\)
and \emph{unconstrained} if
\(
U_j=\mu_s+aj.
\) We define
\[
\operatorname{Max}(C_j)
=
C_s-dj-eL_j,
\]
and
\[\operatorname{Min}(C_j) =C_s-dj-eU_j.\]

\begin{lemma}
\label{lem:column-threshold}
There exists some
\(
j^{\dagger}
\)
such that
\(j\leq j^{\dagger} \text{are unconstrained} \)
and
\(j\geq j^{\dagger}\text{is constrained}.\)
\begin{proof}
The budgeted upper bound
\(B_j= \left\lfloor \frac{s-d(\mu_s+aj)}{ah} \right\rfloor\)
has slope \(-\frac{d}{ah},\)
while the unconstrained bound
\(\mu_s+aj\)
has slope $a$. $a>- \frac{d}{ah}.$ Therefore, once
\(B_j<\mu_s+aj,\)
this inequality will continue to hold for all larger $j$.
\end{proof}
\end{lemma}

\begin{lemma}
\[\operatorname{Max}(j) \equiv C_s-dj \pmod e. \]
\end{lemma}

\begin{proof}
This is immediate from the formula
\[ \operatorname{Max}(j) = C_s-dj-eL_j. \]

\end{proof}
\section{The Maximal Delta}
\label{sec:maximaldelta}

We next generalize the result of $M$ in embedding dimension three, determining the maximal element for any numerical monoid genrerated by generalized arithmetic sequences.
The argument uses two results concerning Betti elements.

The following result of Chapman et al.~\cite{Chapman} relates
the maximal delta of a numerical monoid to the maximal delta
among its Betti elements.

\begin{lemma}[Chapman et al.~\cite{Chapman}]
For a numerical monoid $S$,
\[
\max(\Delta(S))
=
\max_{n\in\operatorname{Betti}(S)}
\left(\max(\Delta(n))\right).
\]
\end{lemma}

We will also use the following description of the Betti elements
for numerical monoids generated by generalized arithmetic
sequences, due to Omidali and Rahmati~\cite{Rahmati}.

\begin{lemma}[Omidali and Rahmati~\cite{Rahmati}]
Set
\[
r \equiv a-1 \pmod{k}
\qquad\text{and}\qquad
p=\frac{a-1-r}{k}.
\]
Then the Betti elements occur in the form
\[
pn_k+n_j,
\qquad r+1\leq j\leq k.
\]
\end{lemma}

We now apply these two results to determine the maximal delta.
By the first lemma, it suffices to determine the maximal delta
among the Betti elements described in the second lemma.

Write
\(
a-1=pk+r
\)
and set
\(
q=\left\lceil\frac{a}{k}\right\rceil.
\)
Since $1<k<a$, we have
\(
q=p+1.
\)

By the second lemma, the Betti elements are
\(
b_j=pn_k+n_j, r+1\leq j\leq k.
\)
Write
\[
j=r+i,
\qquad 1\leq i\leq k-r.
\]

We first consider $j=r+1$, so that $i=1$. Then
\(
b_{r+1}=pn_k+n_{r+1}.
\) This has length $p+1.$
However, since
\(
pk+(r+1)=a,
\)
\[
\begin{aligned}
b_{r+1}
&=p(ah+kd)+(ah+(r+1)d))\\
&=(p+1)ah+ad\\
&=qha+ad\\
&=(qh+d)a.
\end{aligned}
\]
This implies $b_{r+1}$ can also have a factorization using only copies of $a$, of length $qh+d$.
The difference between these two factorization
lengths is
\[
(qh+d)-q=q(h-1)+d.
\]
Thus,
\[
q(h-1)+d
=
\left\lceil\frac{a}{k}\right\rceil(h-1)+d
\]
occurs as a delta element of a Betti element.
We now must show that these two lengths are consecutive.

We begin with some notation:
   $s= ax+ \sum_{i=1}^{k}(ah+id)y_i= ax+ ah\sum_{i=1}^ky_i+ d\sum_{i=1}^k iy_i.$

We define the sum $ \sum_{i=1}^k y_i =Y, \  \sum_{i=1}^kiy_i=W.$ So, $s=ax+ahY+dW.$ We also note that
$s=a(qh+d)$

Proof. We begin with the initial equality:$$a(qh+d) = ax + ahY + dW$$Taking both sides modulo $a$:$$dW \equiv 0 \pmod a$$Given that $\gcd(a, d) = 1$, $a$ must divide $W$. Thus, there exists an integer $m$ such that:$$W = ma$$Substituting $W = ma$ back into the main equation:$$a(qh+d) = ax + ahY + d(ma)$$Dividing both sides by $a$:$$qh + d = x + hY + dm$$ To isolate the quantity $x + Y,$ which represents total length, we rewrite $hY$ as $(h- 1)Y + Y$:$$qh + d = x + Y + h(Y - 1) + dm$$ $$x + Y = qh + d - dm - Y(h - 1)$$3. If $m = 0$, then $W = ma = 0$, which implies $Y = 0$. Substituting $m = 0$ and $Y = 0$ into the expression gives:$$x + Y = x=  qh + d$$Because the length is minimized with $-dm = 0,$ $qh + d$ represents the maximum possible value. For $m \ge 1$, we use the relation $W \le kY$ (relation comes since $i \neq k$ always):$$ma \le kY \implies Y \ge \frac{am}{k}$$Given $q = \left\lceil \frac{a}{k} \right\rceil$, it follows that $Y \ge q$ for $m \ge 1.$ Substituting the lower bound $Y \ge q$ back into our isolated expression for $x + Y$:$$x + Y = qh + d(1 - m) - Y(h - 1) \ge qh + d(1 - m) - q(h - 1)$$Simplifying the right-hand side:$$qh + d(1 - m) - qh + q = d(1 - m) + q$$Thus, for $m \ge 1$:$$x + Y \ge d(1 - m) + q$$ Because no integer values of $m$ exist strictly between the bound for $m = 1$ ($q$) and the bound for $m = 0$ ($qh + d$), these lengths are therefore consecutive.
It now remains to show that no other Betti element produces a larger
delta.  For
\[
n_k=ah+kd
\qquad\text{and}\qquad
n_j=ah+jd,
\]
we have
\[
\begin{aligned}
b_j
&=p(ah+kd)+(ah+jd)\\
&=(p+1)ah+(pk+j)d.
\end{aligned}
\]
Since $j=r+i$ and $pk+r=a-1$,
\[
pk+j=pk+r+i=a-1+i.
\]
Therefore,
\[
b_j=qha+(a-1+i)d.
\]
Rewriting this expression gives
\[
\begin{aligned}
b_j
&=(qh+d-h)a+\bigl(ha+(i-1)d\bigr)\\
&=(qh+d-h)a+n_{i-1},
\end{aligned}
\]
where $n_0=a$.

Thus $b_j$ has a factorization of length
\[
qh+d-h+1.
\]
Comparing this with the factorization of length $q$ gives
\[
qh+d-h+1-q
=
(q-1)(h-1)+d.
\]
Since $h>1$,
\[
(q-1)(h-1)+d
\leq
q(h-1)+d.
\]
Hence these Betti elements cannot produce a delta larger than
\[
q(h-1)+d.
\]

The remaining Betti elements described by Omidali and Rahmati
arise from relations of the form
\[
n_i+n_j=n_u+n_v.
\]
If $s \in S= n_i+n_j,$ it can also be written as $s = hn_0+n_{i+j},$ where $h$ is the number of copies of $a.$
Thus, $s$ either has length difference $h-1 \ or \ 2-2=0$. In either case, the length is strictly less than  $q(h-1)+d$. 

We have therefore shown that
\[
\max_{n\in\operatorname{Betti}(S)}\max(\Delta(n))
=
q(h-1)+d.
\]
Applying the first lemma and recalling that
$q=\left\lceil a/k\right\rceil$, we obtain the following theorem.

\begin{theorem}[Maximal Delta]
For
\[
S_h=\langle a,ah+d,ah+2d,\ldots,ah+kd\rangle,
\]
where $\gcd(a,d)=1$, $1<k<a$, and $h >1$,
\[
\boxed{
\max(\Delta(S_h))
=
\left\lceil\frac{a}{k}\right\rceil(h-1)+d
}.
\]
\end{theorem}

\begin{corollary}
\medskip
Omidali \cite{Omidali} proved that the catenary degree is of the form

 \[ \left\lceil \frac{a}{k} \right\rceil (h)+d\]

 Combining that with our theorem, we get that
 
 \[ max(\Delta(S)) + \left\lceil \frac{a}{k} \right\rceil = c(S),\]

\end{corollary}
which additionally provides a new result in regards to calculating the catenary degree.

\section{Boundary Family}
This section proves that the set \[\{M-re, r\geq 0,\} \subset \Delta(S_h)\]
\label{sec:arithmetic}

\begin{lemma}[Boundary Difference and Index Mapping]

\label{lem:boundary-difference-index-mapping}

Suppose that the columns $C_j$ and $C_{j+1}$ (defined in Section 6) are both unconstrained and that
\[\operatorname{Min}(j)>\operatorname{Max}(j+1).\]
Let
\[m_j=\mu_s+aj, \qquad q=\left\lceil\frac{a}{k}\right\rceil. \]
Define
\[\widetilde r(m) = q-\left\lceil\frac{m+a}{k}\right\rceil+m, \qquad m\geq 0.\]
Then
\[\operatorname{Min}(j)-\operatorname{Max}(j+1)=M-e\widetilde r(m_j).\]
Moreover, the sequence $\{\widetilde r(m)\}_{m\geq 0}$ satisfies:
\begin{enumerate}
    \item $\widetilde r(0)=0$;
    \item $\widetilde r(m+1)-\widetilde r(m)\in\{0,1\}$ for every $m\geq0$;
    \item $\widetilde r(m)\to\infty$ as $m\to\infty$.
\end{enumerate}
Consequently,
\(
\{\widetilde r(m):m\geq0\}=\mathbb N_0.
\)
\end{lemma}

\begin{proof}
Since $C_j$ and $C_{j+1}$ are unconstrained, their upper and lower
endpoints are given by
\(
U_j=m_j \ \text{and} \ 
L_{j+1}
=
\left\lceil\frac{m_j+a}{k}\right\rceil.
\)
Therefore,
\begin{align*}
\operatorname{Min}(j)-\operatorname{Max}(j+1)
&=
\bigl(C_s-dj-eU_j\bigr)
-
\bigl(C_s-d(j+1)-eL_{j+1}\bigr)\\
&=
d+e(L_{j+1}-U_j).
\end{align*}
Since
\(
M=qe+d,
\)
this introduces a similarity with the form $M-re.$
Set
\(
M-re=d+e(L_{j+1}-U_j).
\)
Equating the two expressions and cancelling $d$ gives
\(qe-re=e(L_{j+1}-U_j),
\)
and hence
\[
r=q-L_{j+1}+U_j.
\]
Substituting the expressions for $L_{j+1}$ and $U_j$ yields
\(r = q-\left\lceil\frac{m_j+a}{k}\right\rceil+m_j = \widetilde r(m_j). \)
Thus
\[\operatorname{Min}(j)-\operatorname{Max}(j+1) = M-e\widetilde r(m_j). \]

It remains to establish the stated properties of $\widetilde r$.

First, \(\widetilde r(0) = q-\left\lceil\frac{a}{k}\right\rceil = 0,\)
by the definition of $q$.

Next, for every $m\geq0$,
\begin{align*}
\widetilde r(m+1)-\widetilde r(m)
&=
\left[
q-\left\lceil\frac{m+1+a}{k}\right\rceil+(m+1)
\right]
-
\left[
q-\left\lceil\frac{m+a}{k}\right\rceil+m
\right]\\
&=
1-
\left(
\left\lceil\frac{m+a+1}{k}\right\rceil
-
\left\lceil\frac{m+a}{k}\right\rceil
\right).
\end{align*}
Since increasing the numerator of a ceiling by $1$ can increase the
ceiling by either $0$ or $1$, we have
\[\left\lceil\frac{m+a+1}{k}\right\rceil - \left\lceil\frac{m+a}{k}\right\rceil \in\{0,1\}.\]
Consequently,
\(\widetilde r(m+1)-\widetilde r(m)\in\{0,1\}. \)

Finally, using
\[ \left\lceil x\right\rceil<x+1, \]
we obtain
\begin{align*}
\widetilde r(m)
&=
q-\left\lceil\frac{m+a}{k}\right\rceil+m\\
&>
q-\left(\frac{m+a}{k}+1\right)+m\\
&=
m\left(1-\frac1k\right)
+
\left(q-\frac ak-1\right).
\end{align*}
Since $k\geq2$, we have $1-\frac1k>0$, and therefore
\[\widetilde r(m)\longrightarrow\infty \qquad\text{as }m\to\infty. \]

Thus $\widetilde r(m)$ is a nondecreasing, unbounded sequence of
nonnegative integers, beginning at $0$, whose successive increments
are at most $1$. It follows that the sequence cannot skip any
nonnegative integer. This becomes important when studying the entire subset generated by \[M-re\]
\end{proof}

\begin{lemma}[Consecutiveness of Boundary Extrema]
\label{lem:consecutive-boundary-extrema}
Suppose that the columns $C_j$ and $C_{j+1}$ are both unconstrained and
that
\[\operatorname{Min}(j)>\operatorname{Max}(j+1).\]
Then $\operatorname{Max}(j+1)$ and $\operatorname{Min}(j)$ are
consecutive elements of $L(s)$.
\end{lemma}

\begin{proof}
Every element of $C_j$
is at least $\operatorname{Min}(j)$, while every element of
$C_{j+1}$ is at most $\operatorname{Max}(j+1)$. Thus neither $C_j$ nor
$C_{j+1}$ contains an element strictly between
$\operatorname{Max}(j+1)$ and $\operatorname{Min}(j)$.

Now consider a column $C_r$ with $r<j$. By the budget column lemma (6.1.),
every column between $C_r$ and $C_{j+1}$ is unconstrained. In
particular, the column $C_r$ lies strictly to the left of $C_j$, and
the ordering of \(U_j\) gives
\[\operatorname{Min}(r)>\operatorname{Min}(j).\]
Hence every element of $C_r$ is strictly greater than
$\operatorname{Min}(j)$.

On the other hand, if $r>j+1$, the decrease of \(L_j\) implies
\[\operatorname{Max}(r)<\operatorname{Max}(j+1).\]
Therefore every element of $C_r$ is strictly smaller than
$\operatorname{Max}(j+1)$.

Combining these observations, no element of $L(s)$ lies strictly
between $\operatorname{Max}(j+1)$ and $\operatorname{Min}(j)$. Since
both endpoints belong to $L(s)$, they are consecutive elements of
$L(s)$.
\end{proof}

\begin{proposition}[Boundary Family]
\label{prop:boundary-family}
We have
\[\left\{M-re:r\geq0,\ M-re>0\right\}\subseteq \Delta(S_h).\]
\end{proposition}

\begin{proof}
Fix $r\geq0$ such that
\[M-re>0.\]
By Lemma~\ref{lem:boundary-difference-index-mapping}, the map
$\widetilde r:\mathbb N_0\to\mathbb N_0$ is surjective. Hence there
exists $m\geq0$ such that
\[\widetilde r(m)=r.\]
If we then choose a corresponding $\mu_s$ in the residue class, and then choose an element $s$ large enough so that both
 \(C_j, \ C_{j+1}\)are both unconstrained, Lemma~\ref{lem:boundary-difference-index-mapping}
gives
\(\operatorname{Min}(j)-\operatorname{Max}(j+1)=M-e\widetilde r(m)=M-er.\)
 $M-er>0$, and hence
\(\operatorname{Min}(j)>\operatorname{Max}(j+1).\)
Lemma~\ref{lem:consecutive-boundary-extrema} therefore implies that
$\operatorname{Max}(j+1)$ and $\operatorname{Min}(j)$ are consecutive
elements of $L(s)$. Their difference is consequently an element of
$\Delta(S_h)$, so
\[M-er\in\Delta(S_h).\]

Since $r\geq0$ with $M-er>0$ was arbitrary, we conclude that
\[\left\{M-re:r\geq0,\ M-re>0 \right\} \subseteq \Delta(S_h).\]
\end{proof}

\section{Euclidean Differences and the Three-Gap Theorem}
\label{sec:Euclidean}
\begin{proposition}

\textnormal{Suppose first that}
\[e\mid d.\]
\textnormal{We write}
\[d=ce.\]
Then
\[E(e,d) =\{e,2e,\ldots,ce\}.\]
\end{proposition}
\begin{proof}

Set
\[r=q+c-i.\]
Then
\[
\begin{aligned}
M-re
&=
M-(q+c-i)e \\
&=
qe+ce-(q+c-i)e \\
&=
ie.
\end{aligned}
\]
\end{proof}
Now define
\[\rho=d\bmod e,\]
so that
\[0<\rho<e.\]

Let
\[\rho\equiv d\pmod e, \qquad 0<\rho<e.\]
For \(J\geq 1\), define
\[P_J=\{j\rho\bmod e:0\leq j<J\}.\]
Let \(G_J\) denote the set of cyclic gaps between consecutive elements of \(P_J\).

\begin{proposition}[Three-Gap Realization]
\label{prop:euclidean-gap}
By the Three-Gap Theorem and the subtractive Euclidean algorithm,
\[
\bigcup_{J\geq 1}G_J=E(e,\rho).
\]
Alternatively, every gap $g\in G_J $ created by inserting a new point $j$ and then finding $j\rho \pmod e$ splits the current gap between $j's$ so that it conforms to the subtractive Euclidean Algorithm.
Since \(E(e,d)=E(e,\rho)\), it is therefore enough to prove that every
\[
g\in G_J
\]
belongs to \(\Delta(S_h)\).
\end{proposition}

\begin{lemma}

\[
\forall s\in S, s=aC_s+d\mu_s.
\]
Setting
\(\mu_s=0,\),
\(
s=aC_s.
\)

\[\ell(j,y)\equiv C-dj\pmod e.\]
Since
\[d\equiv \rho\pmod e,\]
we have
\[C-dj\equiv C-j\rho\pmod e.\]
\end{lemma}

\begin{lemma}
\textnormal{Now choose \(j_0\), and consider the consecutive columns}
\[j_0,j_0+1,\ldots,j_0+J-1.\]
\textnormal{After translating all residues by the fixed quantity \(C-dj_0\), the residue classes in these columns are exactly the elements of \(P_J\). Therefore, since \(g\in G_J\), there exist}
\[r_1,r_2\in\{0,\ldots,J-1\}\]
\textnormal{such that, with
\[j_1=j_0+r_1,\qquad j_2=j_0+r_2,\]
the residue classes corresponding to \(j_1\) and \(j_2\) are consecutive and have cyclic separation \(g\). Our goal is to "lift" this modular difference into a difference of actual lengths.}
\end{lemma}

\begin{lemma}
Let \(R_i\in\{0,\ldots,e-1\}\) denote the residue of \(C-dj_i\) modulo \(e\). Thus
\[C-dj_i=R_i+eQ_i\]
for some \(Q_i\in\mathbb Z\). \textnormal{Since the two residue classes have cyclic separation \(g\), we may write
\[R_2-R_1=g+te\]
for some} \(t\in\mathbb Z\).

For a lattice point \((j_i,y_i)\), we have
\[\ell_i = C-dj_i-ey_i = R_i-e(y_i-Q_i). \]
Set
\[
\begin{aligned}
y_i'&=y_i-Q_i,\\
\ell_i&=R_i-ey_i',\\
\ell_2-\ell_1
&=(R_2-R_1)-e(y_2'-y_1').
\end{aligned}
\]
Using
\(R_2-R_1=g+te,
\)
we see that
\[
\ell_2-\ell_1=g
\quad\text{when}\quad
y_2'-y_1'=t.
\]
Equivalently,
\[
y_2-y_1=t+Q_2-Q_1.
\]
Set
\[
\alpha=t+Q_2-Q_1.
\]

Since
\[
Q_i=\frac{C-dj_i-R_i}{e},
\]
we obtain
\[
\begin{aligned}
Q_2-Q_1
&=\frac{-d(j_2-j_1)-(R_2-R_1)}{e}\\
&=\frac{-d(r_2-r_1)-g-te}{e}\\
&=-q(r_2-r_1)
-\frac{\rho(r_2-r_1)+g}{e}-t,
\end{aligned}
\]
where we have used
\[
j_2-j_1=r_2-r_1,
\qquad
R_2-R_1=g+te,
\qquad
d=qe+\rho.
\]
The residue relation between the two consecutive points gives
\[
w=\frac{\rho(r_2-r_1)+g}{e}\in\mathbb Z.
\]
Therefore
\[
Q_2-Q_1=-q(r_2-r_1)-w-t,
\]
and hence
\[
\boxed{\alpha=-q(r_2-r_1)-w.}
\]

In particular, the \(t\)-terms cancel. Moreover,
\[
|r_2-r_1|\leq J-1
\qquad\text{and}\qquad
|w|\leq J-1.
\]
Consequently,
\[
\boxed{|\alpha|\leq(q+1)(J-1).}
\]
\end{lemma}

\begin{lemma}
We now choose \(C\) sufficiently large so that
\[
C\geq(ah+d)(j_0+J-1).
\]
This would then allow
\[
U_j=aj
,\]
which would further generate the allowable \(y\)-coordinates in column \(j\) to be precisely
\[
I_j=\left[\left\lceil\frac{aj}{k}\right\rceil,aj\right]\cap\mathbb Z.
\]

If we then choose \(j_0\) sufficiently large, there exist allowable lattice points
\[
(j_1,y_1),
\qquad
(j_2,y_2)
\]
satisfying
\[
y_2-y_1=\alpha.
\]
\end{lemma}

\begin{proposition}[\text{Realization of Three-Gap Differences}]
By the definition of \(\alpha\),
\[
y_2'-y_1'=t.
\]
Therefore
\[
\begin{aligned}
\ell_2-\ell_1
&=(R_2-R_1)-e(y_2'-y_1')\\
&=(g+te)-et\\
&=g.
\end{aligned}
\] \textbf{Thus} \(g\) \textbf{is a difference of factorization lengths.}

\textnormal{Finally, because the residue classes corresponding to \(j_1\) and \(j_2\) are consecutive elements of \(P_J\), there is no residue class between them. Hence there is no factorization length strictly between the two constructed lengths. Therefore, they are consecutive elements of the length set \(L(s)\), and so}
\[g=\ell_2-\ell_1\in\Delta(S_h).\]

Since \(g\in G_J\) was arbitrary, we conclude that
\[G_J\subseteq\Delta(S_h)\]
for every \(J\geq1\).
\end{proposition}

\begin{corollary}
\label{corollary:Euclidean-Inclusion}
\[
E(e,\rho)=\bigcup_{J\geq1}G_J\subseteq\Delta(S_h).
\]
Finally, since
\[E(e,d)=E(e,\rho),\]
we obtain
\[
\boxed{E(e,d)\subseteq\Delta(S_h).}
\]
\end{corollary}

This completes the proof.
\section{Truncated Columns}
\label{sec:otherdelta}
\begin{lemma}[Truncated-Column Lemma]
\label{lem:trunc}
Let $x<y$ be consecutive elements of $\mathcal L(s)$ with
\[
x\in C_i,\qquad y\in C_j.
\]
If
\[
y-x>e,
\]
then
\[
x=\operatorname{Max}(i),\qquad
y=\operatorname{Min}(j),\qquad
j<i.
\]
\end{lemma}

\begin{proof} Suppose first that
\(
x<\operatorname{Max}(i).
\)
Since $C_i$ is an arithmetic progression of common difference $e$,
the element $x+e$ also belongs to $C_i$. Since $y-x>e$, we have
\(x<x+e<y,\)
contradicting the assumption that $x$ and $y$ are consecutive.
Therefore
\(
x=\operatorname{Max}(i).
\) Similarly, if
\(
y>\operatorname{Min}(j),
\)
then $y-e\in C_j$ and
\(
x<y-e<y,
\)
again contradicting consecutiveness. Hence
\(y=\operatorname{Min}(j).\)

\(\operatorname{Max}(i)=x<y=\operatorname{Min}(j)
\leq \operatorname{Max}(j).\)
But this is impossible when $i<j$. Consequently,
\(
j<i.
\)
\end{proof}

\begin{corollary}[Non-adjacent Bound]
\label{cor:gapbound}
Let $x<y$ be consecutive elements of $\mathcal L(s)$ with
\[
x\in C_i,\qquad y\in C_j,\qquad i<j.
\]
Then
\[
0<y-x\leq e.
\]
\end{corollary}

\begin{proof}
Clearly $y-x>0$. If $y-x>e$, Lemma~\ref{lem:trunc} would imply
$j<i$, contradicting $i<j$. Hence
\(
0<y-x\leq e.
\)
\end{proof}
We next consider the cases of non-adjacent columns with consecutive lengths. Our goal is to force these columns to be consecutive in nature
For $j\geq1$, define
\(
m_j=\left\lceil\frac{dj}{e}\right\rceil
\)
and
\(
o_j=em_j-dj.
\), therefore measuring the "error", or how far $dj$ is from a multiple of $e$
Thus
\[
o_j=(-dj)\bmod e\in\{0,1,\ldots,e-1\}.
\]

\begin{lemma}[Survival of Truncated Columns]

\label{lem:survival}
Let $x<y$ be consecutive elements of $\mathcal L(s)$ with
\[
x\in C_i,\qquad y\in C_j,\qquad J=j-i\geq1.
\]
Then, for every $1\leq t\leq J$,
\[
L_{i+t}\leq \ell_i-m_t\leq U_{i+t}.
\]
Consequently,
\[
x+o_t
=
C_s-d(i+t)-e(\ell_i-m_t)
\in C_{i+t}\subseteq\mathcal L(s).
\]
\end{lemma}

\begin{proof}
The identity follows directly from \[
\begin{aligned}
C_s-d(i+t)-e(\ell_i-m_t)
&=
C_s-di-e\ell_i-dt+em_t\\
&=
x+o_t.
\end{aligned}
\]
The inequality $L_{i+t} \leq \ell_i-m_t \leq U_{i+t}$ ensures that $ x+ o_t$ is an allowable value, and it therefore remains necessary to prove that $ \ell_i -m_t$ is within the allowable length values.

We first prove the upper bound.

Since
\(
\ell_i\leq U_i\leq\mu+ai,
\)
we have
\(
\ell_i-m_t<\mu+a(i+t),
\)
and hence
\[
\ell_i-m_t\leq\mu+a(i+t). \tag{10.3}
\]

Also,
\(
\ell_i\leq U_i\leq B_i.
\)
Furthermore,
\[
\begin{aligned}
B_i-B_{i+t}
&=
\left\lfloor\frac{C_s-di}{h}\right\rfloor
-
\left\lfloor\frac{C_s-di-dt}{h}\right\rfloor\\
&\leq
\left\lfloor\frac{dt}{h}\right\rfloor+1\\
&\leq
\left\lceil\frac{dt}{e}\right\rceil\\
&=
m_t.
\end{aligned}
\]
Therefore
\[
B_i-m_t\leq B_{i+t},
\]
and hence
\[
\ell_i-m_t
\leq B_i-m_t
\leq B_{i+t}. \tag{9.4}
\]
Combining (9.3) and (9.4) yields
\[
\ell_i-m_t\leq U_{i+t}.
\]

We now prove the lower bound.

Since $x<y$,
\(
-dJ+e(\ell_i-\ell_j)>0.
\)
Thus
\(
e(\ell_i-\ell_j)>dJ,
\)
so
\(
\ell_i-\ell_j
\geq
\left\lceil\frac{dJ+1}{e}\right\rceil.
\)
Because
\(
\left\lceil\frac{dJ+1}{e}\right\rceil
\geq
\left\lceil\frac{dJ}{e}\right\rceil
=m_J,
\)
we obtain
\(
\ell_i-m_J\geq\ell_j.
\)
Since $L_r$ is nondecreasing in $r$
we have
\(
L_{i+J}=L_j\leq\ell_j.
\)
Hence
\(
L_{i+J}\leq\ell_i-m_J.
\) More generally, for $1\leq t\leq J$, the same argument gives
\(
\ell_i-m_t\geq L_{i+t}.
\)
Therefore
\[
L_{i+t}\leq\ell_i-m_t\leq U_{i+t},
\]
\end{proof}

\begin{proposition}[Non-adjacent Columns]
\label{prop:nonadj}
Let $x<y$ be consecutive elements of $\mathcal L(s)$ with
\[x\in C_i,\qquad y\in C_j,\qquad i<j.
\]
Set
\(
J=j-i,\qquad \delta=y-x.
\)
Then
\(
0<\delta\leq e.
\)

If $e\nmid d$, then
\[
\delta
=
\min
\left\{
(-dt)\bmod e:
1\leq t\leq J,\ e\nmid dt
\right\}.
\]

In particular, $\delta$ is the first positive gap above $0$ in the
finite rotation set
\[
P_J=\{(-dt)\bmod e:0\leq t\leq J\},
\]
and hence
\[
\boxed{\delta\in E(e,d).}
\]
\end{proposition}

\begin{proof}
The inequality
\[
0<\delta\leq e
\]
follows from Corollary~\ref{cor:gapbound}.

Suppose that
\[
e\nmid d.
\]
For $1\leq t\leq J$, Lemma~\ref{lem:survival} gives
\[
x+o_t\in\mathcal L(s),
\]
where
\[
o_t=(-dt)\bmod e.
\]
Since $x$ and $y$ are consecutive and $x+o_t\geq x$, we must have
either
\(
o_t=0
\qquad\text{or}\qquad
x+o_t\geq y.
\)
Equivalently,
\[
o_t=0
\qquad\text{or}\qquad
o_t\geq\delta. \tag{9.5}
\]

On the other hand,
\[
\delta
=
y-x
=
-dJ+e(\ell_i-\ell_j),
\]
so
\[
\delta\equiv-dJ\pmod e.
\]
Thus
\[
o_J=\delta\bmod e. \tag{9.6}
\]

Since $e\nmid d$, the previous argument shows that
\[
\delta<e.
\]
Consequently, by (10.6),
\[
o_J=\delta>0.
\]
Combining this with (9.5), we obtain
\[
o_t=0
\qquad\text{or}\qquad
o_t\geq\delta
\]
for every $1\leq t\leq J$, while $o_J=\delta$. Hence
\[
\delta
=
\min
\{o_t:1\leq t\leq J,\ o_t>0\}.
\]
Since
\[
o_t=(-dt)\bmod e,
\]
this gives
\[
\delta
=
\min
\left\{
(-dt)\bmod e:
1\leq t\leq J,\ e\nmid dt
\right\}.
\]

Thus $\delta$ is the first positive gap above $0$ in the finite rotation
set
\[
P_J=\{(-dt)\bmod e:0\leq t\leq J\}.
\]
By Proposition~\ref{prop:euclidean-gap}, every gap in such a finite
rotation set belongs to $E(e,d)$. Therefore
\[
\delta\in E(e,d).
\]
\end{proof}

\begin{corollary}[Non-adjacent Gaps Are Euclidean]
\label{cor:nonadj-euclidean}
If $x<y$ are consecutive elements of $\mathcal L(s)$ lying in
non-adjacent columns
\[
i<j,\qquad j-i\geq2,
\]
then
\[
\boxed{y-x\in E(e,d).}
\]
\end{corollary}

\begin{proof}
This is an immediate consequence of Proposition~\ref{prop:nonadj}.
\end{proof}

\section{Forward Containment}
\label{sec:forwardcontainment}
\begin{theorem}[Forward Containment]
\label{thm:forward}
\[\Delta(S_h)\subseteq E(e,d) \cup \left\{M-re:r\geq0,\ M-re>0 \right\}.\]
\end{theorem}
\begin{proof}
Let
\[ \delta\in\Delta(S_h). \]
Then there exist consecutive factorization lengths
\[ \ell_1<\ell_2 \]
in some length set $\mathcal{L}(s)$ such that
\[ \delta=\ell_2-\ell_1. \]

Let $P$ and $Q$ be the corresponding lattice points in columns
$i$ and $j$, respectively.

\medskip

\noindent
\textbf{Case 1: Same column.}

If $i=j$, then the two lengths differ by exactly $e$, since
successive elements in a column differ by $e$. Therefore,
\[ \delta=e\in E(e,d). \]

\medskip

\noindent
\textbf{Case 2: Adjacent columns.}

Suppose
\[j=i+1.\]

By Lemma~\ref{lem:boundary-difference-index-mapping}, if
\[P=\operatorname{Min}(i),\qquad Q=\operatorname{Max}(i+1),\]
then
\[\delta =M-re\]
for some $r\geq0$

\medskip

\noindent
\textbf{Case 3: Non-adjacent columns.}

If
\[
j-i\geq2,
\]
then 9.4. gives
\[
\delta
\in
E(e,d)
\cup
\left\{
M-re:r\geq0,\ M-re>0
\right\}.
\]

Therefore,
\[
\Delta(S_h)
\subseteq
E(e,d)
\cup
\left\{
M-re:r\geq0,\ M-re>0
\right\}.
\]
\end{proof}


\section{Main Theorem}
\label{sec:maintheorem}
\begin{theorem}[Delta Set Decomposition]
\label{thm:maintheorem}
\[
\boxed{
\Delta(S_h)
=
E(e,d)
\cup
\left\{
M-re:r\geq0,\ M-re>0
\right\}
}
\]
where
\[
e=h-1,
\qquad
M=
\left\lceil\frac{a}{k}\right\rceil(h-1)+d.
\]
\end{theorem}

\begin{proof}
Theorem~\ref{thm:forward} gives the forward containment
\[
\Delta(S_h)
\subseteq
E(e,d)
\cup
\left\{
M-re:r\geq0,\ M-re>0
\right\}.
\]

For the reverse containment, Proposition~\ref{prop:boundary-family} gives
\[\left\{M-re:r\geq0,\ M-re>0\right\}\subseteq
\Delta(S_h),
\]
while Proposition~\ref{corollary:Euclidean-Inclusion} gives
\[E(e,d)\subseteq\Delta(S_h).\]

Combining these two containments yields
\[\boxed{\Delta(S_h)=E(e,d)\cup\left\{M-re:r\geq0,\ M-re>0\right\}.}\]
\end{proof}

\section{Bounds and Structural Relationships}

\begin{proposition}
\label{prop:basic-delta-values}
The values $d$ and $h-1$ belong to $\Delta(S_h)$.
\end{proposition}

\begin{proof}
This proof proceeds trivially.
Since
\[d,h-1\in E(d,h-1),\]
and by the Euclidean containment established earlier,
\[E(d,h-1)\subseteq \Delta(S_h),\]
it follows that
\[d,h-1\in\Delta(S_h).\]
\end{proof}


\begin{corollary}[Case $h=2$]
If $h=2$, then
\(e=h-1=1,
\)
and hence
\[
E(M,h-1)=E(M,1)=\{1,2,\ldots,M\}.
\]
\end{corollary}

\begin{theorem}[GCD Reduction]
\label{thm:gcd-reduction}
Let
\[
g=\gcd(h-1,d),
\]
and define
\[
h'=1+\frac{h-1}{g},
\qquad
d'=\frac{d}{g}.
\]
Then
\[
\boxed{
\Delta(S_h)=g\,\Delta(S_{h'}).
}
\]
\end{theorem}

\begin{proof}
Set
\(
e=h-1.
\)
Then
\(
e=ge', d=gd',
\)
where
\(
e'=h'-1=\frac{h-1}{g}.
\)
Since the Euclidean algorithm is homogeneous, we have
\[
E(d,e)
=
E(gd',ge')
=
gE(d',e').
\]
Now consider the arithmetic component. Set
\[
M=\left\lceil\frac{a}{k}\right\rceil e+d.
\]
Using
\[
e=ge',
\qquad
d=gd',
\]
we obtain
\[
M
=
\left\lceil\frac{a}{k}\right\rceil ge'+gd'
=
g\left(
\left\lceil\frac{a}{k}\right\rceil e'+d'
\right).
\]
Define
\[
M'
=
\left\lceil\frac{a}{k}\right\rceil e'+d'.
\]
Then
\[
M=gM'.
\]
Consequently,
\[
M-re
=
gM'-rge'
=
g(M'-re').
\]
Therefore,
\[
\left\{
M-re>0 : r\geq 0
\right\}
=
g
\left\{
M'-re'>0 : r\geq 0
\right\}.
\]
By the description of the delta set,
\[
\Delta(S_h)
=
E(d,e)
\cup
\left\{
M-re>0 : r\geq 0
\right\}.
\]
\newpage
Substituting the identities above gives
\(
\Delta(S_h)
=
gE(d',e')
\cup
g
\left\{
M'-re'>0 : r\geq 0
\right\}.
\)
Factoring out \(g\), we obtain
\(
\Delta(S_h)
=
g
\left(E(d',e')\cup\left\{M'-re'>0 : r\geq 0\right\}\right).\)
Again by the description of the delta set, the expression
inside the parentheses is  \(\Delta(S_{h'}).\)
Hence\[\boxed{\Delta(S_h)=g\,\Delta(S_{h'}).}\]
\end{proof}

\section{Minimum}
\label{sec:Minimum}
The proof for the minimum of the Delta Set in relation to generalized arithmetic sequences as a simple consequence of a theorem proven by Garcia, Omar, O'Neill, and Yih \cite{GarciaOmarONeillYih}

Let \[n=\ \langle n_1,n_2,n_3,...n_k \rangle\] be minimally generated. 
\medskip
Then, \[min(\Delta(S) = gcd(n_k-n_{k-1},n_{k-1} - n_{k-2}...,n_2-n_1)\]

Applying this theorem to the numerical monoid we are working with, we yield
\[min(\Delta(S) = gcd(ah+kd-\langle ah+(k-1)d),ah+(k-1)d - (ah+(k-2)d)) ..., (ah+d) - a \rangle\]
\[min(\Delta(S) = gcd(d,d ..., (ah+d) - a )\]
\[min(\Delta(S) = gcd(d,d ..., a(h-1)+d )\]
\[min(\Delta(S) = gcd(d,d ..., a(h-1) )\]
However, because\[gcd (a,d)=1,\] the preceding expression simplifies to \[gcd(d,h-1)\]

\section{Fibonacci Sequence Generation}
First, we demonstrate how to embed Fibonacci sequences in delta sets.
Set
\[
\left\lceil \frac{a}{k} \right\rceil = 2.
\]
Then
\[
M = 2(h-1) + d.
\]

Set
\[
h-1 = e.
\]
Thus,
\[
M = 2e + d.
\]

Let
\[
e = F_{n-2}, \qquad d = F_{n-3}.
\]

Since
\[
F_n = F_{n-1} + F_{n-2}
\]
and
\[
F_{n-1} = F_{n-2} + F_{n-3},
\]
we obtain
\[
F_n = 2F_{n-2} + F_{n-3}.
\]

Therefore,
\[
M = 2e + d = F_n.
\]
 $\Delta(S_h) =E(M,e)$ can also be rewritten as $E(d,e) \cup \{M-re \ | \ M-re > 0, r\geq 0\}.$ Applying this framework, the set $E(e,d)$ gives all elements of the Fibonacci sequence up
to $F_{n-2}$. Meanwhile, the sequence
\(
M - re
\)
gives
\(
M = F_n
\)
and
\(
M - e = F_n - F_{n-2} = F_{n-1}.
\)

Taking the union, we obtain the entire Fibonacci set
\[
\{1, 2, 3, 5, 8, \ldots, F_{n-3}, F_{n-2}, F_{n-1}, F_n\},
\]
where
\[
F_n = M.
\]

\section{ Metallic Fibonacci Sequences}
Our next idea is to see if we can generalize these results to those with $k>1$ in the sequence $F_{n}= kF_{n-1}+F_{n-2}.$ This was firstly computationally experimented with until a relevant conjecture could be extracted. Using first principles, we then proved this conjecture.
Let
\[
P_n = kP_{n-1} + P_{n-2}.
\]

\begin{theorem}
There exist $k$ metallic Fibonacci sequences of the same order
(i.e., with the same $k$) in the delta set. Equivalently,
\[
E(P_n, P_{n-1})
=
\bigcup_{i=2}^{n}
\bigcup_{j=0}^{k-1}
\left\{ Q_i^j \right\}.
\]
\end{theorem}

We define $Q_i^j = P_i - jP_{i-1}$ for a fixed $j$.

\begin{definition}
    The \emph{Subtractive Euclidean Algorithm}, which we denote as $E(x,y)$, can be recursively defined in this way: 
    1. Let $x>y$. 
    2. Set $x=ky+r$. 
    Then $E(x,y) = \{x-jy: 0 \leq j \leq k-1 \} \cup E(y,r)$, with termination occurring when both terms are equal.
\end{definition}

Note that the Metallic Fibonacci Sequence occupies a similar form to the Subtractive Euclidean Algorithm. In fact, if we set $x, y, r = P_i, P_{i-1}, P_{i-2}$, then they occupy the exact same form. Therefore, we can apply the subtractive Euclidean Algorithm.
\[
E(P_{i}, P_{i-1}) = \{P_i - jP_{i-1} : 0 \leq j \leq k-1\} \cup E(P_{i-1}, P_{i-2})
\]
Repeating this would get 
\[
E(P_{i-1}, P_{i-2}) = \{P_{i-1} - jP_{i-2} : 0 \leq j \leq k-1\} \cup E(P_{i-2}, P_{i-3})
\]
The recurrence pattern suggests an introduction of a variable. Let $B_i = \{P_i - jP_{i-1} : 0 \leq j \leq k-1\}$.
$E(P_n, P_{n-1})$ then becomes the union of $B$ over all $i$. 
\[
E(P_n, P_{n-1}) = \bigcup_{i=2}^n B_i
\]
For each $j$, we would have the sequence $P_2 - jP_1, P_3 - jP_2, P_4 - jP_3, \dots, P_n - jP_{n-1}$.
If we then rearrange the blocks $B_1, B_2, B_3, \dots, B_n$ so that the value of each block at $j$ is utilized to create a sequence, we then see that that sequence is itself metallic.
For example,
\[
\begin{array}{c|cc}
    & j=0 & j=1 \\
\hline
B_2 & 2  & 1 \\
B_3 & 5  & 3 \\
B_4 & 12 & 7 \\
B_5 & 29 & 17 \\
B_6 & 70 & 41
\end{array}
\]
\[
2 \rightarrow 5 \rightarrow 12 \rightarrow 29 \rightarrow 70 
\]
\emph{This satisfies the Pell Sequence.}
We see here that $x_{j+1} = kx_{j} + x_{j-1}$.
Our goal is to prove that this pattern is not mere coincidence, but a natural result.

Define $Q^{(j)} = P_{i} - jP_{i-1}$.

\begin{proposition}
    $Q_i^j$ is a metallic sequence, and the delta set consists of exactly $k$ of these sequences.
\end{proposition}

\begin{proof}
    Recall that a metallic fibonacci sequence is defined as $P_n = kP_{n-1} + P_{n-2}$.
    In this case, we should prove that this recurrence holds for $Q_i^{j+1}$:
    \begin{align*}
        Q_i^{j+1} &= P_{i+1} - jP_{i} \\
        &= kP_i + P_{i-1} - jP_i \\
        &= kP_{i} + P_{i-1} - j(kP_{i-1} + P_{i-2}) \\
        &= kP_i - jkP_{i-1} + P_{i-1} - jP_{i-2} \\
        &= k(P_i - jP_{i-1}) + Q_{i-1}
    \end{align*}
    \[
    Q_{i+1}^j = kQ^j + Q_{i-1}
    \]
\end{proof}

Because $j$ ranges from $0$ to $k-1$, there are precisely $k$ of these sequences of order $n-1$, which leads us to the above theorem.

Set $P_{n-1} = h-1, P_{n} = M$.
By the above theorem, we get 
\[
E(P_n, P_{n-1}) = \bigcup_{i=2}^n \bigcup_{j=0}^{k-1} \{Q_i^j\}
\]
If we set $P_n = M, P_{n-1} = h-1$, we then get 
\[
E(M,e) = \bigcup_{i=2}^n \bigcup_{j=0}^{k-1} \{Q_i^j\}
\]
However, $E(M,e) = \Delta(S_h)$.
Therefore, 
\[
\Delta(S_h) = \bigcup_{i=2}^n \bigcup_{j=0}^{k-1} \{Q_i^j\}
\]
The significance of this theorem is that by setting $P_n = M, P_{n-1} = e$, one can produce exactly $k$ sequences of order $n-1$ that are \emph{Metallic}.

\section{Catenary Degree Set}

This section provides a proof that under structural conditions, $h+1 \in C(S)$

\begin{theorem}
\[
h+1 \in C(S)
\] 
when $k > 2$.
\end{theorem}
\begin{remark}
    We need $k > 2$ since this ensures that $l(n_i), l(n_j), l(n_{i+j}), l(n_0)$ can occupy different coordinates in the factorization length vector. The indices $i, j$ must satisfy
\(
i+j \leq k
\) so that enough generators can be used for $hn_0+n_{i+j}=n_i+n_j.$
\end{remark}

\begin{proof}
We aim to construct an element $s \in S$ such that
\(
c(s) = h+1.
\)

Let
\[
s = hn_0 + n_{i+j},
\]
where
\[
n_0 = a, \qquad n_i = ah + id,
\]
with $i \neq j$ and
\[
i+j \leq k.
\]

Since
\[
hn_0 + n_{i+j}
=
ha + ah + (i+j)d
\]
and
\[
n_i + n_j
=
(ah + id) + (ah + jd),
\]
we obtain
\[
hn_0 + n_{i+j} = n_i + n_j.
\]

Thus, $s$ has two factorizations.

The first factorization has length
\[
|z| = h+1,
\]
while the second factorization has length
\[
|z'| = 2.
\]

Moreover,
\[
\gcd(z,z') = 0,
\]
so the two factorizations lie in distinct $R$-classes.
Therefore,
\[
c(s) = h+1,
\]
and hence
\[
h+1 \in C(S).
\]
 
\end{proof}

\section{Conjectures Related to the Catenary Degree Sets}

\begin{conjecture}
if \[k \geq 3, k \leq a-2, 2\in C(s)\]
Furthermore, the set \[ [0,h+1] \] contains only \[0,2,h+1\]
\end{conjecture}
\begin{conjecture}
\[\left\lceil \frac{a-k}{k-1} \right\rceil + h+d \in C(S)\]
\end{conjecture}
\begin{conjecture}
For \[S_h =\langle a,ah,ah+d,ah+2d,...ah+kd \rangle,\] 
\[\Delta(S_h) = E(e,d) \]
\end{conjecture}
\section*{Future Work}

Several directions emerge naturally from the structural and computational
results developed in this paper.

\begin{itemize}

    \item \textbf{Catenary Degree Sequences:}
    The relation
    \[
    c(S_h)
    =
    \max\Delta(S_h)
    +
    \left\lceil\frac{a}{k}\right\rceil
    \]
    suggests a deeper connection between delta sets and catenary
    degrees. A following question is whether the Euclidean component
    \[
    E(d,h-1)
    \]
    controls any portion of the catenary degree sequence besides the
    maximal element.

    \item \textbf{Representation-Theoretic Connections:}
    The lattice parametrization used in this paper suggests a possible
    connection with representation theory and the algebraic structures
    associated to numerical semigroups. It would
    be interesting to investigate whether the generators or Gröbner bases of \[I_{S_h}\] explain the Euclidean Component and Arithmetic Sequence

    \item \textbf{Topological Structure of Factorizations:}
    The factorization lattice and its column decomposition also suggest
    a topological perspective. Associating simplicial complexes,
    factorization graphs, or other combinatorial topological objects
    to the factorizations of elements of $S_h$ and ask whether their
    topological invariants can help predict information about the delta sets is one way. In particular, it would be interesting to
    determine whether changes in connectivity or higher-dimensional
    homology correspond to the appearance of particular elements of
    $\Delta(S_h)$.

    \item \textbf{Extension to Other Families of Numerical Monoids:}
    The methods of Euclidean analysis and Lattice Parametrization may have corresponding analogues for other families of numerical monoids, which leads to the question of how these methods may be implemented to study other families.

\end{itemize}

\section{Acknowledgments}

The author would first like to thank Professor Nathan Kaplan for his generosity in reading through the text, supporting the author through submission, and helping with certain proofs. Additionally, the author would like to acknowledge the use of AI for suggestions in regards to the geometric lattice proof, though all suggestions were independently checked and rederived by the author.

Department of Mathematics, University of Chicago.\\
Email address: \texttt{anarayan2030@uchicago.edu}

\vspace{1em}

\noindent
\begin{thebibliography}{99}

\bibitem{ONeill}
C. O'Neill and R. Pelayo,
\textit{Factorization invariants in numerical monoids},
Algebraic and Geometric Methods in Discrete Mathematics
(ed. H. Harrington, M. Omar, M. Wright),
Contemporary Mathematics \textbf{685} (2017), 231--249.

\bibitem{RosalesGarcia}
J. C. Rosales and P. A. Garc\'ia-S\'anchez,
\textit{Numerical Semigroups},
Developments in Mathematics, vol. 20, Springer, New York, 2009.


\bibitem{BowlesChapman}
C. Bowles, S. T. Chapman, N. Kaplan, and D. Reiser,
\textit{On delta sets of numerical monoids},
Journal of Algebra and Its Applications \textbf{5} (2006), no. 5, 695--718.

\bibitem{GarciaLlenaMoscarielloSym}
P. A. Garc\'ia-S\'anchez, D. Llena, and A. Moscariello,
\textit{Delta sets for symmetric numerical semigroups with embedding dimension three},
Communications in Algebra \textbf{45} (2017), no. 11, 4653--4658.

\bibitem{GarciaLlenaMoscarielloNonSym}
P. A. Garc\'ia-S\'anchez, D. Llena, and A. Moscariello,
\textit{Delta sets for nonsymmetric numerical semigroups with embedding dimension three},
Forum Mathematicum \textbf{30} (2018), no. 1, 15--30.

\bibitem{Chapman}
S. T. Chapman, P. A. Garc\'ia-S\'anchez, D. Llena, V. Ponomarenko, and J. C. Rosales,
\textit{On the delta set and the Betti elements of a BF-monoid},
Arabian Journal of Mathematics \textbf{1} (2012), 53--61.


\bibitem{Rahmati}
M. Omidali and F. Rahmati,
\textit{On the type and the minimal presentation of certain numerical subgroups},
Communications in Algebra \textbf{37} (2009), no. 4, 1275--1283.




\bibitem{Omidali}
M. Omidali,
\textit{The catenary and tame degree of numerical monoids generated by generalized arithmetic sequences},
Forum Mathematicum \textbf{24} (2012), no. 3, 627--640.

\bibitem{GarciaOmarONeillYih}
S. Garcia, M. Omar, C. O'Neill, and S. Yih,
\textit{Factorization length distribution for affine semigroups II: Asymptotic behavior for numerical semigroups with arbitrarily many generators},
Journal of Combinatorial Theory, Series A \textbf{178} (2021), 105358.

\end{thebibliography}
\end{document}